\documentclass[a4paper,oneside]{article}
\usepackage[english]{babel}
\usepackage{mathtools}
\usepackage{amsthm}

\usepackage[a4paper, left=3 cm, right=3cm]{geometry}

\usepackage[dvipsnames]{xcolor}
\usepackage[draft, markup=sout, authormarkup=brackets]{changes}
\definechangesauthor[name=Favre, color=Plum]{FG}

\allowdisplaybreaks

\usepackage{amssymb}
\usepackage{enumerate}
\usepackage{enumitem}
\usepackage{braket}
\usepackage{slashed}
\usepackage{bbold}
\usepackage{amssymb}
\usepackage{quoting}
\usepackage[final]{graphicx}
\usepackage{textcomp} 
\usepackage{latexsym}
\usepackage[symbol*,perpage]{footmisc}
\usepackage{wasysym}
\usepackage{cleveref}
\usepackage[square,numbers]{natbib}
\usepackage[final]{pdfpages}
\usepackage{tikz}
\usepackage{subfiles}
\usepackage{lipsum}

\usepackage[title]{appendix}

\usepackage{changes}

\usepackage{xargs}                   
\definechangesauthor[name={Favre}, color=red]{fg}
\usepackage{flowfram} 

\usepackage[textsize=normalsize, color=black!10]{todonotes} 

\newcommandx{\commento}[2][1=]{\todo[linecolor=black,backgroundcolor=black!10,bordercolor=black,#1]{\scriptsize#2}}
\newcommandx{\commentob}[2][1=]{\todo[linecolor=blue,backgroundcolor=blue!25,bordercolor=blue,#1]{\scriptsize#2}}

\newtheoremstyle{break}{\topsep}{\topsep}{\itshape}{}{\bfseries}{}{\newline}{}

\newtheorem{thm}{Theorem}[]
\newtheorem*{thm-nn}{Theorem}
\newtheorem{lem}[thm]{Lemma}
\newtheorem*{lem-nn}{Lemma}

\newtheorem*{cor-nn}{Corollary}

\newtheorem{defn}[thm]{Definition}

\theoremstyle{remark}

\allowdisplaybreaks 

\let\TeXcirc\circ
\def\circ{{\raisebox{1pt}{\mathsurround1pt$\TeXcirc$}}}

\let\TeXtextendash\textendash
\def\textendash{\scalebox{1.3}{{\raisebox{-4.5pt}{\mathsurround0pt$\TeXtextendash$}}}}

\newcommand{\calH}{{\mathcal H}}

\newcommand{\grad}{\nabla_x} 

\newcommand{\R}{\mathbb{R}}

\newcommand{\intv}[1]{\int_{\R^d}#1\,\dv}

\newcommand{\ki}{k_{ij}}
\newcommand{\kj}{k_{ji}}
\newcommand{\sumi}{\sum_{i=1}^{N}}
\newcommand{\sumj}{\sum_{j=1}^{N}}
\newcommand{\sumij}{\sum_{i,j=1}^{N}}

\newcommand{\sprod}[2]{\langle #1 , #2 \rangle} 
\newcommand{\dv}{dv\;}

\newcommand{\derivtintero}{\frac{d}{dt}}

\def\sfT{\mathsf{T}}
\def\sfL{\mathsf{L}}
\def\sfH{\mathsf{H}}
\def\sfA{\mathsf{A}}
\def\eps{\varepsilon}

\usepackage{fancyhdr}
\usepackage{lastpage}
\makeatletter
\renewcommand\tableofcontents{%
    \@starttoc{toc}%
}
\makeatother

\newcommand{\nc}{\normalcolor}

\begin{document}
\nocite{*} 

\def\comment#1{\color{red}#1\color{black}}

\title{\vspace{-2cm}\textbf{Hypocoercivity and fast reaction limit for linear reaction networks with kinetic transport}}

\author{{\Large Gianluca Favre\footnote{Faculty of Mathematics, University of Vienna,	Oskar-Morgenstern-Platz 1, 1090 Wien, Austria  \newline	E-mail: {\tt gianluca.favre@univie.ac.at}} 
\hspace{15pt} and \hspace{15 pt} 
Christian Schmeiser\footnote{Faculty of Mathematics, University of Vienna,	Oskar-Morgenstern-Platz 1, 1090 Wien, Austria \newline E-mail: {\tt christian.schmeiser@univie.ac.at}}}
}
\date{}
\maketitle

\begin{abstract}
\noindent
The long time behavior of a model for a first order, weakly reversible chemical reaction network is considered, where the movement
of the reacting species is described by kinetic transport. The reactions are triggered by collisions with a nonmoving background with constant
temperature, determining the post-reactional equilibrium velocity distributions. Species with different particle
masses are considered, with a strong separation between two groups of light and heavy particles. As an approximation, the
heavy species are modeled as nonmoving. Under the assumption of at least one moving species, long time convergence
is proven by hypocoercivity methods for the cases of positions in a flat torus and in whole space. In the former case
the result is exponential convergence to a spatially constant equilibrium, and in the latter it is algebraic decay to zero, at the 
same rate as solutions of parabolic equations. This is no surprise since it is also shown that the macroscopic (or reaction
dominated) behavior is governed by the diffusion equation.
\end{abstract}

\section{Introduction}

We consider $N$ chemical species $S_1,\ldots, S_N$ with different particle masses moving in a periodic box
or in whole space. The interaction with a stationary background with constant temperature $T$ triggers first order chemical reactions with reaction rates 
independent of the velocity of the incoming particle. The velocity of the outgoing particle is sampled from a Maxwellian distribution 
with parameters taken from the background, i.e.~mean velocity zero and temperature $T$. The resulting reaction network is assumed 
to be connected and weakly reversible, meaning that for each reaction $S_i\to S_j$ there exists a reaction path 
$S_j \to \cdots \to S_i$. 

These assumptions lead to a system of $N$ linear kinetic transport equations for the phase space number densities of the reacting 
species. We shall make the additional assumption that the species can be split into two groups of \emph{light} and \emph{heavy} 
particles, where the particle masses are of comparable size within each group but strongly disparate between the groups. 
A corresponding nondimensionalization of the equations, assuming at least one light species, will suggest a simplified model, where 
the heavy particles do not move. As a result we consider a system of kinetic equations (for the light species) coupled, via the
reaction terms, to a system of ordinary differential equations (pointwise in position space, for the heavy species).

The construction of equilibrium solutions is straightforward. In equilibrium, the position densities are constant, 
and the velocity distributions of the light particles are Maxwellians. The position densities are complex balanced equilibria
of the reaction network. Existence and uniqueness for given total mass are standard results of 
the theory of chemical reaction networks.

Our main results are exponential convergence to equilibrium in the case of the periodic box and algebraic decay to zero
in whole space. In both situations the rates and constants are computable.
Although general results for Markov 
processes imply that relative entropies are nonincreasing \cite{FonJou}, the decay result is not obvious, since the entropy 
dissipation is not coercive relative to the equilibrium. We employ the abstract $L^2$-hypocoercivity method of 
\cite{Schmeiser2,Schmeiser} and its extension to whole space problems \cite{BDMMS}.
The main difficulty is the proof of \emph{microscopic coercivity,} meaning here that the reaction terms without the
transport produce exponential convergence to a \emph{local equilibrium,} where the total number density of all species 
might still depend on position and time. Two alternative proofs are presented. In the first one, relaxation in velocity space
is separated from relaxation to chemical equilibrium and known results for the latter \cite{Fellner} could be used. The second 
proof extends the proof in \cite{Fellner} by introducing reaction paths in species-velocity space. For completeness and 
comparability we fully present both proofs, showing that the second proof never gives a worse result.

The second result is a macroscopic or fast-reaction limit. For length scales large compared to the mean free path between
reaction events and for the corresponding diffusion time scales, the system is in local equilibrium and the total number density 
solves the heat equation.

Systematic approaches to hypocoercivity have been started in \cite{MouNeu, Villani}, where Lyapunov functions 
based on modified $H^1$-norms are constructed. More recently, an approach without smoothness assumptions on initial
data, motivated by \cite{Hereaux}, has been developed in \cite{Schmeiser2, Schmeiser}, see \cite{Desvillettes, Mouhot-fr} for overviews. Recently the latter approach has been extended to the analysis of algebraic decay rates in whole space problems 
\cite{BDMMS}.
Hypocoercivity for systems of kinetic equations coupled by linearized collision terms has been shown in \cite{TU-Zamponi-Ester}. 
For a nonlinear system modeling a second order pair generation-recombination reaction, both hypocoercivity and the fast reaction
limit have been analyzed in \cite{Neumann-Schmeiser}.

This work can be seen as an extension of the corresponding result for linear reaction diffusion models \cite{Fellner},
which has recently been extended to general mass action kinetics \cite{FelTan}, bringing the theory for reaction diffusion models
close to the best results on the global attractor conjecture \cite{GAC-Horn} for ODE models without transport \cite{GAC-Craciun}.

Many extensions of the present results are desirable. Besides the inclusion of collision effects and of second order
reactions, questions of energy and momentum balance pose significant challenges, where a trade-off between mathematical 
manageability and modeling precision has to be found. One goal is the rigorous justification of the derivation of reaction diffusion
systems from kinetic models as an extension of results for linear cases \cite{BisDes}. 

Finally, we describe the structure of the rest of this article. In the following section the kinetic model is formulated
including a dimensional analysis and the reduction to a system with partially nonmoving species. The formal macroscopic 
limit is presented and our main results on the long term behavior of solutions and on the rigorous justification of the macroscopic 
limit are formulated. In Section \ref{sec:micro} our main technical result on 'microscopic coercivity' is proven, i.e.~a spectral gap
for the reaction operator. Sections \ref{sec:hypo} and \ref{sec:macro-limit} are concerned with the proofs of our main results on long 
time behaviour and, respectively, on the rigorous macroscopic limit.

\section{The model -- main results}\label{sec:results}
We denote the chemical species by $S_1,\dots,S_N$ and the reaction constant for the reaction $S_i\to S_j$ by $k_{ji}\ge 0$, 
$i,j=1,\ldots,N$, where $k_{ji}=0$ means that the reaction does not occur.
More completely, also including velocities $v\in\R^d$, we assume that the jump $(S_i,v) \to (S_j,v')$ occurs with rate
constant $k_{ji}M_j(v')$, as described above independent of the incoming velocity, where the Maxwellian distribution is given by
 \begin{equation*}
M_i(v) = \bigg(\frac{2 \pi k_B T}{m_i}\bigg)^{-d/2} \exp \bigg(- \frac{|v|^2 m_i}{2 k_B T} \bigg) \,,
\end{equation*}
with the Boltzmann constant $k_B$, the constant given background temperature $T$, and the particle masses 
$m_1\le \cdots \le m_N$ of the respective species $S_1,\ldots,S_N$. Actually all our results can be proven with $M_1,\ldots,M_N$ 
replaced by arbitrary probability distributions with mean zero and finite fourth order moments.

The phase space number density of species $S_i$ at time $t\ge 0$ is denoted by $f_i(x,v,t)\ge 0$, $i=1,\ldots,N$, with the position 
variable $x$. We consider two cases: 
\begin{enumerate}
\item \emph{Periodic box:} $x\in \mathbb{T}^d$, the flat $d$-dimensional torus, represented by the cube 
$[0,L]^d$ with periodic boundary conditions for $f_1,\ldots,f_N$.
\item \emph{Whole space:} $x\in\R^d$, with $f_1,\ldots,f_N$ integrable, i.e.~a finite total number of particles.
\end{enumerate}
In the following, integrations with respect to $x$ will be written over $\Omega$, where $\Omega = [0,L]^d$ for the periodic box and $\Omega = \R^d$ for whole space.

The phase space distributions satisfy the evolution system
\begin{equation}
\label{eq: k-r equation}
\partial_t f_i + v \cdot \grad f_i = \sumj \big(\ki \rho_j M_i - \kj f_i\big) \,,\qquad i=1,\ldots,N\,,
\end{equation}
where the left hand side describes free transport and the right hand side the chemical reactions with position densities
\begin{equation}\label{def:rho}
    \rho_j(x,t) = \intv{f_j(x,v,t)} \,,
\end{equation}
where we will sometimes also use the notation $\rho_{f,j}$ to avoid ambiguity.
We assume that there are $N_l\ge 1$ light species $S_1,\ldots,S_{N_l}$ and $N-N_l$ heavy species $S_{N_l+1},\ldots,S_N$.
The separation of the two groups is expressed in the assumption 
$$
  \mu:= \frac{m_{N_l}}{m_{N_l+1}} \ll 1 \,.
$$
In a nondimensionalization we introduce as reference velocity the thermal velocity 
$$
  v_{th} := \sqrt{\frac{k_B T}{m_{N_l}}}
$$
of the heaviest light species $S_{N_l}$. As reference time $\overline t$ we choose an average value of $k_{ij}^{-1}$, $i,j = 1,\ldots, N$.
The reference length is given by $\overline x = \overline t \,v_{th}$. After the nondimensionalization
$$
   v \to v_{th} v \,,\quad t \to \overline t t \,,\quad x \to \overline x x \,,\quad f_i \to (\overline x v_{th})^{-d} f_i \,,\quad M_i \to v_{th}^{-d} M_i\,,
   \quad k_{ij} \to \frac{k_{ij}}{\overline t} \,,\quad L \to \overline x L \,,
$$
the equations \eqref{eq: k-r equation}, \eqref{def:rho} look the same, but with
\begin{equation}\label{Mi-scaled}
   M_i(v) =  \left(2 \pi \theta_i\right)^{-d/2} \exp \left(- \frac{|v|^2}{2\theta_i} \right) \,,\qquad \theta_i = \frac{m_{N_l}}{m_i} \,,
   \qquad i = 1,\ldots,N\,.
\end{equation}
In particular we have $\theta_i \ge 1$, $i=1,\ldots,N_l$, for the light particles and $\theta_i = O(\mu)$, $i=N_l + 1,\ldots,N$,
for the heavy particles, such that $M_i(v) \to \delta(v)$, $i=N_l + 1,\ldots,N$, in the distributional sense as $\mu\to 0$. In this
limit it is consistent to also look for solutions, where the heavy particles are nonmoving, i.e.~$f_i(x,v,t) = \rho_i(x,t)\delta(v)$, 
$i=N_l + 1,\ldots,N$. Therefore, for the rest of this work we shall consider the system
\begin{eqnarray}
  \partial_t f_i + v \cdot \grad f_i &=& \sumj \big(\ki \rho_j M_i - \kj f_i\big) \,,\qquad i=1,\ldots,N_l\,,\label{eq:k-r}\\
  \partial_t \rho_i &=& \sumj \big(\ki \rho_j - \kj \rho_i\big) \,,\qquad i=N_l+1,\ldots,N\,,\label{eq:r}
\end{eqnarray}
with $N_l\ge 1$ and with $M_i$ given by \eqref{Mi-scaled}, subject to initial conditions
\begin{equation}
  f_i(x,v,0) = f_{I,i}(x,v) \,,\quad \rho_j(x,0) = \rho_{I,j}(x) \,,\quad x\in\Omega\,,\, v\in\R^d\,,\, i \le N_l\,,\,
    j > N_l \,,
\end{equation}
with initial data satisfying $f_{I,i}\in L_+^1(\Omega\times\R^d)$, $\rho_{I,j}\in L_+^1(\Omega)$. For simplicity of notation 
we formally set $f_i(x,v,t) = \rho_i(x,t)M_i(v)$ with $M_i(v) = \delta(v)$, $i> N_l$, and write the system \eqref{eq:k-r},
\eqref{eq:r} in the equivalent form
\begin{equation}\label{eq:full}
   \partial_t f_i + \sigma_i \,v \cdot \grad f_i = \sumj \big(\ki \rho_j M_i - \kj f_i\big) \,,\qquad i=1,\ldots,N\,,
\end{equation}
with $\sigma_i = 1$ for $i\le N_l$ and $\sigma_i = 0$ otherwise. We shall consider initial value problems with
\begin{equation}\label{eq:IC}
  f(x,v,0) = f_I(x,v) \,,\qquad x\in\Omega \,,\quad v\in\R^d \,.
\end{equation}

The system \eqref{eq:full} conserves the total number of particles: The total position density 
$$
  \rho(x,t) := \sum_{i=1}^N \rho_i(x,t)
$$
satisfies
\begin{equation}\label{mass-cons}
  \partial_t \rho + \nabla_x\cdot \left( \sum_{i=1}^{N_l} \int_{\R^d} v f_i \,dv\right) = 0 \,,
\end{equation}
and therefore
$$
  \int_{\Omega} \rho(x,t)dx = {\textbf M} := \sum_{i=1}^{N_l} \int_{\Omega} \int_{\R^d} f_{I,i}(x,v) dv\,dx 
   + \sum_{j=N_l+1}^N \int_\Omega\rho_{I,j}(x)dx \,,\qquad t\ge 0\,.
$$

\subsection*{Local and global equilibria}
\begin{defn}[Local equilibrium]
A state $f(x,v,t)=(f_1(x,v,t), \ldots, f_N(x,v,t))$ is called a local equilibrium for \eqref{eq:full}, if it balances the reactions, i.e., if
\begin{equation}\label{def: equilibrio}
\sumj \big(\ki \rho_j M_i - \kj f_i\big) =0\,,\qquad i=1,\ldots,N\,.
\end{equation}
\end{defn}

The set of all local equilibria can be described in terms of properties of the directed graph with nodes $S_1,\ldots,S_N$ and
edges $S_i \to S_j$, when $k_{ji}>0$. Roughly speaking, there is a simple characterization of local equilibria, if the graph has enough edges.

\noindent{\bf Assumption A1:} The directed graph corresponding to the reaction network is \emph{connected} and \emph{weakly reversible,} 
which means that for each pair $(i,j)\in \{1,\ldots,N\}^2$ there exists  a \emph{path} $(j=i_0,i_1,\ldots,i_{P_{ij}}=i)$ such that 
$k_{i_p i_{p-1}}>0$, $p=1,\ldots,P_{ij}$.

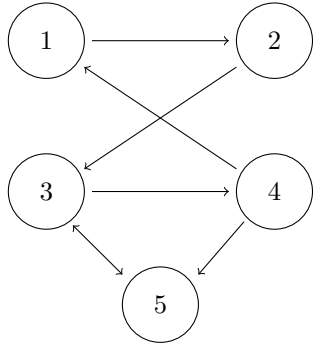
\begin{figure}\begin{center}
\begin{tikzpicture}
	\draw (0.5,3.5)circle (0.5);
	\draw (3.5,3.5)circle (0.5);
	\draw (0.5,1.5)circle (0.5);
	\draw (3.5,1.5)circle (0.5);
	\draw (2,0)circle (0.5);
	
	\draw [->] (1.1,3.5)--(2.9,3.5);
	\draw [->] (1.1,1.5)--(2.9,1.5);
	\draw [->] (3,3.15)--(1,1.8);
	\draw [<-] (1,3.15)--(3,1.8);
	\draw [->] (3.1,1.1)--(2.5,0.4);
	\draw [<->] (0.85,1.05)--(1.5,0.4);
	
	\node at (0.5,3.5) {1};
	\node at (3.5,3.5) {2};
	\node at (0.5,1.5) {3};
	\node at (3.5,1.5) {4};
	\node at (2,0) {5};
\end{tikzpicture}\end{center}
\caption{A connected and weakly reversible reaction network.
Examples for shortest path lengths: $P_{14} = 1$, $P_{52} = 2$, $P_{25} = 4$, 
the latter with the path $(5,3,4,1,2)$.}
\end{figure}

An example is given in Figure 1. Note that the path from $j$ to $i$ is in general not unique. For the following a fixed choice of a path of minimal length $P_{ij}$ is used for each pair $(i,j)$. This also means that paths are not self-intersecting
in the sense that each reaction $S_{i_{p-1}}\to S_{i_p}$ appears only once.

\begin{lem}\label{lem:NL}
Let Assumption A1 hold. Then every local equilibrium is of the form $\rho(x,t)F(v)$ with
$$
  F_i(v) = \eta_i M_i(v) \,,\quad i = 1,\ldots,N \,,
$$
where $\rho(x,t)$ is arbitrary and $(\eta_1,\ldots,\eta_N)$ is the unique solution of
\begin{equation}\label{eq:eta}
   \sumj \big(\ki \eta_j - \kj \eta_i\big) = 0 \,,\quad i = 1,\ldots,N \,,\qquad \sum_{i=1}^N \eta_i = 1 \,,
\end{equation}
satisfying $\eta_i > 0$, $i=1,\ldots,N$.
\end{lem}

\begin{proof}
A first consequence of Assumption A1 is that for each $i\in \{1,\ldots,N_l\}$ there exists at least one $k_{ji}>0$ and at least one 
$k_{ij}>0$. This implies that for a local equilibrium $f_i(x,v,t) = \rho_i(x,t)M_i(v)$, $i=1,\ldots,N_l$, must hold. Therefore
$$
  \sumj \big(\ki \rho_j - \kj \rho_i\big) =0\,,\quad i=1,\ldots,N \,.
$$
Now it is a standard result of reaction network theory (see, e.g., \cite{Fellner,Horn}), in our simple case of first order reactions
related to the Perron-Frobenius theorem, that the connectedness and weak reversibility imply 
that there is a one-dimensional solution space spanned by $(\eta_1,\ldots,\eta_N)$, where all components have the same sign.
In the language of reaction network theory these are \emph{complex balanced equilibria}.
\end{proof}

A \emph{global equilibrium} is a local equilibrium, which is also a steady state solution of \eqref{eq:k-r}, \eqref{eq:r}
compatible with conservation of total mass. Since at least one equation has a transport term, the function $\rho$
from Lemma \ref{lem:NL} has to be constant for a global equilibrium. In the case $\Omega=\R^d$ we expect dispersion 
and consequential decay to zero. Therefore a nontrivial global equilibrium is only defined for $\Omega=\mathbb{T}^d$ by 
\begin{equation}\label{def:f_infty}
   f_\infty(x,v) := \rho_\infty F(v) \,, \qquad\text{with } \rho_\infty := \textbf{M}/L^d \,.
\end{equation}

\subsection*{Microscopic coercivity -- convergence to equilibrium}
We write the system \eqref{eq:full} in the abstract form
\begin{equation}\label{eq:abstract}
\partial_t f + \sfT f = \sfL f \,,
\end{equation}
with the \emph{transport operator} $\sfT$ and the \emph{reaction operator} $\sfL$ defined as 
\begin{equation}\label{def:TL}
  (\sfT f)_i = \sigma_i\, v\cdot\nabla_x f_i \,, \qquad 
       (\sfL f)_i = \sumj \big(\ki \rho_j M_i - \kj f_i \big) \,,\qquad i=1,\ldots N\,.
\end{equation}
Lemma \ref{lem:NL} characterizes the nullspace of the reaction operator. A projection to this nullspace is given by
\begin{equation}\label{def:Pi}
  \Pi f :=  \rho F \,,\qquad
  \text{with } \rho = \sumj \rho_j \,.
\end{equation}
It is easily seen that $\Pi$ is a projection and that $\langle \Pi f,g\rangle = \langle \Pi f,\Pi g\rangle$, which implies that $\Pi$ is orthogonal.
Since the application of the projection involves integration with respect to $v$ and summation over all species,
we do not only have $\sfL\Pi=0$, but the mass conservation property of the collision operator can be written as 
$\Pi\sfL = 0$.

Considering the quadratic relative entropy with respect to the local equilibrium $F$ suggests the introduction of the
weighted $L^2$-space $\calH$ with the scalar product
\begin{equation}\label{scalar-prod}
  \langle f,g \rangle := \sum_{i=1}^N \int_{\Omega}\int_{\R^d} \frac{f_i g_i}{F_i} dv\,dx \,,
\end{equation}
and with the induced norm $\|\cdot\|$. In the case $\Omega=\mathbb{T}^d$, the members of $\calH$ are periodic with respect to $x$.
Note that, for $i>N_l$, we have $f_i(x,v) = \rho_{f,i}(x)\delta(v)$, 
$g_i(x,v) = \rho_{g,i}(x)\delta(v)$, $F_i(x,v) = \eta_i\delta(v)$, and it has to be understood that 
$$ 
  \int_{\Omega}\int_{\R^d} \frac{f_i g_i}{F_i} dv\,dx = \int_{\Omega}\int_{\R^d}\frac{\rho_{f,i}(x)\rho_{g,i}(x)}{\eta_i} \delta(v) dv\,dx
  = \int_{\Omega} \frac{\rho_{f,i}(x)\rho_{g,i}(x)}{\eta_i} dx \,,\qquad i> N_l\,. 
$$
Our main technical result, which will be 
proved in the following section, is coercivity of the reaction operator with respect to its null space. This property will be called 
\emph{microscopic coercivity.}

\begin{lem}\label{lem:micro}
Let Assumption A1 hold. Then $\Pi$ defined by \eqref{def:Pi} is the orthogonal projection to the nullspace of the reaction operator 
$\sfL: \calH \to \calH$ defined in \eqref{def:TL}. Furthermore there exists a constant $\lambda_m>0$
such that 
$$
   -\langle \sfL f,f \rangle \ge \lambda_m \|(1-\Pi)f\|^2 \,.
$$
\end{lem}

This lemma is one of the main tools in the proofs of our results on the long time behavior, presented in Section \ref{sec:hypo}:

\begin{thm}\label{thm:torus}
Let Assumption A1 hold and let $\Omega=\mathbb{T}^d$.  
Then there exist constants $C,\lambda>0$ such that for every $f_I\in\calH$ and with $f_\infty$ given by \eqref{def:f_infty},
the solution $f$ of \eqref{eq:full}, \eqref{eq:IC} satisfies
$$
  \|f(\cdot,\cdot,t) - f_\infty\|^2 \le Ce^{-\lambda t} \|f_I - f_\infty\|^2 \,.
$$
\end{thm}

\begin{thm}\label{thm:Rd}
Let Assumption A1 hold and let $\Omega=\R^d$.  
Then for every $f_I\in\calH \cap L^1(dv\,dx)$ there exists a constant $C>0$ such that the solution $f$ of \eqref{eq:full}, \eqref{eq:IC} satisfies
$$
  \|f(\cdot,\cdot,t)\|^2 \le C(1+t)^{-d/2}\,.
$$
\end{thm}

\subsection*{Macroscopic (fast reaction) limit}

We introduce a diffusive macroscopic rescaling $x\to x/\eps$, $t\to t/\eps^2$ with $0<\eps\ll 1$. Note, however, that
in the case $\Omega=\mathbb{T}^d$ we still consider a fixed $\eps$-independent torus after the rescaling.
The abstract form \eqref{eq:abstract}
of our system becomes
\begin{equation}\label{eq:eps}
  \eps^2 \partial_t f_\eps + \eps \sfT f_\eps = \sfL f_\eps \,,
\end{equation}
with a now $\eps$-dependent solution $f_\eps$. Assuming convergence to $f_0$ as $\eps\to 0$, the formal limit of the equation
implies that $f_0$ is a local equilibrium, i.e.~$f_0(x,v,t) = (\Pi f_0)(x,v,t) = \rho_0(x,t)F(v)$. It remains to determine $\rho_0$.
The rescaled microscopic part $R_\eps := \frac{(1-\Pi)f_\eps}{\eps}$ satisfies 
\begin{equation}\label{kinetic/eps}
  \eps\partial_t f_\eps + \sfT f_\eps = \sfL R_\eps
\end{equation}
with the formal limit 
\begin{equation}\label{eq:R0}
   \sfT f_0 = \sfL R_0 \,. 
\end{equation}
Finally, we also need the conservation law
$$
  \partial_t \Pi f_\eps + \Pi\sfT \frac{f_\eps}{\eps} =  \partial_t \Pi f_\eps + \Pi\sfT \frac{\Pi f_\eps}{\eps} + \Pi\sfT R_\eps = 0 \,,
$$
and observe that the diffusive scaling is consistent, since the \emph{diffusive macroscopic limit} identity
\begin{equation}\label{diff-macro}
  \Pi\sfT\Pi = 0 
\end{equation}
holds, which is easily checked, since $\Pi$ maps to a vector of centered Maxwellians, $\sfT$ provides a factor $v$,
and the second application of $\Pi$ involves an integration with respect to $v$. The property \eqref{diff-macro}
will also be essential in the proof of decay to equilibrium in Section \ref{sec:hypo} and
it guarantees the necessary
solvability condition $\Pi\sfT f_0 = 0$ for \eqref{eq:R0}. Substituting its solution into the limiting conservation law should provide
the missing information on $\rho_0$:
\begin{equation}\label{abstr-diff}
  \partial_t \Pi f_0 + \Pi\sfT\widehat{\sfL}^{-1}\sfT f_0 = 0 \,,
\end{equation}
where $\widehat{\sfL}$ denotes the restriction of $\sfL$ to $(1-\Pi)\calH$.

In order to translate the abstract result, we first note that
$$
  (\sfT\Pi f_0)_i = \sigma_i \eta_i M_i \,v\cdot \nabla_x\rho_0 \,,\qquad i=1,\ldots,N \,. 
$$
Since application of $\Pi$ involves integration with respect to $v$ as first step, the identity $\int_{\R^d} vM_i \,dv = 0$ implies 
\eqref{diff-macro}. The solution of \eqref{eq:R0} is then given by
$$
  R_{0,i} = \left(\widehat{\sfL}^{-1}\sfT f_0 \right)_i = - \frac{\eta_i M_i}{K_i} v\cdot\nabla_x\rho_0 \,,\qquad K_i = \sumj k_{ji} \,,
  \qquad i\le N_l \,.
$$ 
A straightforward computation gives
$$
  \Pi\sfT\widehat{\sfL}^{-1}\sfT f_0 = -D\Delta_x \rho_0  \,F \,,\qquad D = \sum_{i=1}^{N_l} \frac{\eta_i \theta_i}{K_i} \,.
$$
Thus, \eqref{abstr-diff} is equivalent to the diffusion equation
\begin{equation}\label{heat-eq}
  \partial_t \rho_0 = D\Delta_x \rho_0 \,.
\end{equation}
The following result, providing a rigorous justification of the formal asymptotics, will be proved in Section \ref{sec:macro-limit}.
It also relies on the microscopic coercivity result Lemma \ref{lem:micro}.

\begin{thm}\label{thm:macro}
Let Assumption A1 and $f_I\in\calH$ hold (with $\Omega=\R^d$ or $\Omega=\mathbb{T}^d$). Then the solution $f_\eps$ of \eqref{eq:IC}, \eqref{eq:eps} converges, as $\eps\to 0+$,
to $\rho_0 F$ in $L_{loc}^\infty(\R^+;\calH)$ weak *, where $\rho_0\in L^\infty(\R^+;L^2(\Omega))$ is a distributional solution
of \eqref{heat-eq} subject to the initial condition $\rho_0(x,0) = \intv{f_I(x,v)}$.
\end{thm}

\section{Microscopic coercivity (proof of Lemma \ref{lem:micro})}\label{sec:micro}

We shall give two different proofs of the coercivity result. Both are inspired by the proof of the corresponding result in
\cite{Fellner}. The first approach is based on a splitting between the species and velocity spaces, where for the former the result of \cite{Fellner} can be used directly. In the second approach the reaction paths of Assumption A1 are extended to paths in the larger species-velocity 
space. In both cases the coercivity constant $\lambda_m$ can in principle be computed
explicitly. Since the computations are rather based on algorithms than on explicit formulas, a comparison of the results
for both approaches would be difficult.

For the following computations we introduce $U_i := \frac{f_i}{\eta_i M_i}$, $V_i := \frac{g_i}{\eta_i M_i}$ and rewrite the 
reaction operator as
$$
  \sfL f = \sumj M_i \left( k_{ij} \eta_j \int M_j' U_j' dv' - k_{ji}\eta_i U_i\right)  \,,
$$
where the primes indicate evaluation at $v'$. This implies
\begin{eqnarray*}
  \langle \sfL f,g\rangle &=& \sum_{i,j=1}^N k_{ij} \eta_j \int_{\Omega\times\R^d\times\R^d} M_i M_j' U_j' V_i \,dv'\,dv\,dx 
  -\frac{1}{2} \sum_{i,j=1}^N k_{ji} \eta_i \int_{\Omega\times\R^d} M_i U_i V_i \,dv\,dx\\
   && -\frac{1}{2} \sum_{i,j=1}^N k_{ji} \eta_i \int_{\Omega\times\R^d\times\R^d} M_i M_j' U_i V_i \,dv'\,dv\,dx  \,.
\end{eqnarray*}
Now \eqref{eq:eta} is used in the second term on the right hand side and the change of variables $(i,v)\leftrightarrow (j,v')$
in the third:
\begin{eqnarray*}
  \langle \sfL f,g\rangle = -\frac{1}{2} \sum_{i,j=1}^N k_{ij} \eta_j \int_{\Omega\times\R^d\times\R^d} M_i M_j' \left(U_i V_i + U_j' V_j'
  - 2 U_j' V_i\right) dv'\,dv\,dx 
\end{eqnarray*}
This shows that $\sfL$ can only be expected to be symmetric in the case of \emph{detailed balance}, i.e.~when 
$k_{ij}\eta_j = k_{ji}\eta_i$ for all $i,j=1,\ldots,N$. It also shows negative semi-definiteness of $\sfL$:
\begin{equation}\label{Lff}
\langle \sfL f,f\rangle = -\frac{1}{2} \sum_{i,j=1}^N k_{ij}\eta_j \int_{\Omega\times\R^d\times\R^d} M_i M_j' (U_i - U_j')^2\,dv'\,dv\,dx
\le 0 \,.
\end{equation}

\paragraph{First proof -- separation of species and velocity contributions:} 
The strategy is to split the dissipation term into contributions measuring the deviation from Maxwellian velocity distributions
on the one hand, and from reaction equilibria of the position densities on the other hand:
\begin{eqnarray}
  - \langle \sfL f,f\rangle &=& \frac{1}{2} \sum_{i,j=1}^N k_{ij}\eta_j \int_{\Omega\times\R^d\times\R^d} M_i M_j' 
  \left[U_i - \frac{\rho_i}{\eta_i} + \frac{\rho_i}{\eta_i} - \frac{\rho_j}{\eta_j} + \frac{\rho_j}{\eta_j} - U_j'\right]^2\,dv'\,dv\,dx 
  \nonumber\\
  &=& \sumi \left( \sumj \frac{k_{ij} \eta_j^2 + k_{ji} \eta_i^2}{2\eta_i\eta_j}\right) \int_{\Omega\times\R^d} 
    \frac{(f_i - \rho_i M_i)^2}{\eta_i M_i} dv\,dx \nonumber\\
    &&  + \sumij \frac{k_{ij}\eta_j}{2} \int_{\Omega} \left( \frac{\rho_i}{\eta_i} - \frac{\rho_j}{\eta_j}\right)^2 dx \,. \label{split1}
\end{eqnarray}
The norm of the microscopic part of the distribution is split correspondingly:
\begin{eqnarray}
  \|(1-\Pi)f\|^2 &=& \sumi \int_{\Omega\times\R^d} \frac{(f_i - \rho_i M_i + \rho_i M_i - \rho \eta_i M_i)^2}{\eta_i M_i} dv\,dx 
  \nonumber\\
  &=& \sumi \int_{\Omega\times\R^d} \frac{(f_i - \rho_i M_i)^2}{\eta_i M_i} dv\,dx +
  \sumi \int_{\Omega} \frac{(\rho_i - \rho \eta_i)^2}{\eta_i} dx \,.    \label{split2}
\end{eqnarray}
On the one hand the connectedness of the reaction network implies
$$
    \min_{1\le i\le N} \sumj \frac{k_{ij} \eta_j^2 + k_{ji} \eta_i^2}{2\eta_i\eta_j} =: \gamma_1 > 0 \,,
$$
which allows to relate the first terms on the right hand sides of \eqref{split1} and \eqref{split2}.
On the other hand \cite[eqn. (2.15)]{Fellner} could be used for the second terms. For comparability of the 
results of the two proofs we include the derivation of this second inequality. We start with the relation
$$
   \sumi \int_\Omega \frac{(\rho_i - \rho\eta_i)^2}{\eta_i} dx
   = \frac{1}{2}  \sumij \eta_i\eta_j \int_\Omega \left( \frac{\rho_i}{\eta_i} - \frac{\rho_j}{\eta_j}\right)^2 dx \,,
$$
which is easily verified by adding and subtracting $\rho$ in the parenthesis on the right hand side, expanding the
square, and using that $\rho$ is the total density. For each pair $(i,j)$ we use Assumption A1 and choose a path of minimal
length $P_{ij}$ from $j$ to $i$, which implies
\begin{eqnarray*}
   \sumi \int_\Omega \frac{(\rho_i - \rho\eta_i)^2}{\eta_i} dx
   &=& \frac{1}{2}  \sumij \eta_i\eta_j \int_\Omega \left( \sum_{p=1}^N \left( \frac{\rho_{i_p}}{\eta_{i_p}} 
   - \frac{\rho_{i_{p-1}}}{\eta_{i_{p-1}}}\right)\right)^2 dx \\
   &\le& \frac{1}{2}  \sumij \eta_i\eta_j P_{ij} \sum_{p=1}^N \int_\Omega \left(  \frac{\rho_{i_p}}{\eta_{i_p}} 
   - \frac{\rho_{i_{p-1}}}{\eta_{i_{p-1}}}\right)^2 dx \,,
\end{eqnarray*}
by an application of the Cauchy-Schwarz inequality. With the definition
$$
  \mu_{ij} := \min_{1\le p \le P_{ij}} k_{i_p i_{p-1}}\eta_{i_{p-1}} > 0\,,
$$
we have the further estimate
\begin{eqnarray*}
   \sumi \int_\Omega \frac{(\rho_i - \rho\eta_i)^2}{\eta_i} dx
   &\le& \frac{1}{2}  \sumij \frac{\eta_i\eta_j P_{ij}}{\mu_{ij}} \sum_{p=1}^N k_{i_p,i_{p-1}}\eta_{i_{p-1}}
   \int_\Omega \left(  \frac{\rho_{i_p}}{\eta_{i_p}} - \frac{\rho_{i_{p-1}}}{\eta_{i_{p-1}}}\right)^2 dx \\
   &\le& \frac{1}{\gamma_2} \sumij \frac{k_{ij}\eta_j}{2} \int_\Omega \left(  \frac{\rho_i}{\eta_i} - \frac{\rho_j}{\eta_j}\right)^2 dx\,,
\end{eqnarray*}
with 
\begin{equation}\label{gamma2}
    \frac{1}{\gamma_2} := \sumij \frac{\eta_i\eta_j P_{ij}}{\mu_{ij}} \,.
\end{equation}
In the second inequality above we have used that each pair $(i_p,i_{p-1})$ occurs only once in a reaction path of minimal 
length.
This concludes the proof of microscopic coercivity with $\lambda_m = \min\{\gamma_1,\gamma_2\}$.

\paragraph{Second proof -- species-velocity space paths:} Starting from the representation \eqref{Lff} and using that the path from $j$ to $i$ is not self-intersecting we have
\begin{eqnarray*}
 - \langle \sfL f,f\rangle &\ge& \frac{\mu_{ij}}{2} \sum_{p=1}^{P_{ij}} \int_{\Omega\times\R^d\times\R^d} 
   M_{i_p} M_{i_{p-1}}' (U_{i_p} - U_{i_{p-1}}')^2 \,dv'\,dv\,dx \\
  &=& \frac{\mu_{ij}}{2} \sum_{p=1}^{P_{ij}} \int_{\Omega\times\R^{(P_{ij}+1)d}} 
   \prod_{q=0}^{P_{ij}} M_{i_q}(v_q) \,(U_{i_p}(v_p) - U_{i_{p-1}}(v_{p-1}))^2 \,dv_0 \ldots dv_{P_{ij}}\,dx \,.
\end{eqnarray*}
With the Cauchy-Schwarz inequality on $\R^{P_{ij}}$ this can be estimated further by
\begin{eqnarray}
 - \langle \sfL f,f\rangle &\ge& \frac{\mu_{ij}}{2P_{ij}}  \int_{\Omega\times\R^{(P_{ij}+1)d}} 
   \prod_{q=0}^{P_{ij}} M_{i_q}(v_q) \left(\sum_{p=1}^{P_{ij}} \left(U_{i_p}(v_p) - U_{i_{p-1}}(v_{p-1})\right)\right)^2 \,dv_0 
     \ldots dv_{P_{ij}}\,dx \nonumber\\
   &=& \frac{\mu_{ij}}{2P_{ij}}  \int_{\Omega\times\R^{(P_{ij}+1)d}} 
   \prod_{q=0}^{P_{ij}} M_{i_q}(v_q)  \left(U_i(v_{P_{ij}}) - U_j(v_0)\right)^2 \,dv_0 \ldots dv_{P_{ij}}\,dx \nonumber\\
   &=& \frac{\mu_{ij}}{2P_{ij}}  \int_{\Omega\times\R^{2d}} M_i M_j'  \left(U_i - U_j'\right)^2 \,dv'\,dv\,dx \,.\label{micro-est1}
\end{eqnarray}
As indicated at the beginning of this section, the strategy in these estimates was to extend the path $i_0,\ldots,i_{P_{ij}}$
in the species space $\{1,\ldots,N\}$ to all possible paths of the form $(i_0,v_0),\ldots, (i_{P_{ij}},v_{P_{ij}})$ in the species-velocity space
$\{1,\ldots,N\} \times \R^d$.

As the next step we observe that
\begin{eqnarray*}
 && \sum_{i,j=1}^N \int_{\Omega\times\R^{2d}} \eta_i \eta_j M_i M_j' (U_i - U_j')^2 dv'\,dv\,dx \\
 && = \sum_{i,j=1}^N \int_{\Omega\times\R^{2d}} \eta_i \eta_j M_i M_j' \left( (U_i - \rho)^2 + (U_j' - \rho)^2 
  -  2(U_i - \rho)(U_j' - \rho) \right)dv'\,dv\,dx \\
 && = 2\|(1-\Pi)f\|^2 \,.
\end{eqnarray*}
Combining this with \eqref{micro-est1} completes the alternative proof of Lemma \ref{lem:micro} with 
$\lambda_m = \gamma_2$, as defined in \eqref{gamma2}. The result of the second proof is always at least as good
as that of the first.

\section{Hypocoercivity}\label{sec:hypo}

Quantitative results on the decay to equilibrium will be shown by employing the hypocoercivity approach of \cite{Schmeiser}
with modifications introduced in \cite{BDMMS}. 

In the case of the torus, $\Omega=\mathbb{T}^d$, we assume w.l.o.g.~$\rho_\infty = \textbf{M} = 0$, which can always be achieved 
by a redefinition of the solution. Thus, in both cases $\Omega = \mathbb{T}^d$ and $\Omega = \R^d$ we expect $f\to 0$ as 
$t\to\infty$. The functional $f\mapsto \|f\|^2$ can then be understood as relative entropy and a natural candidate for a Lyapunov
function. However, by the obvious antisymmetry of the transport operator $\sfT$,
\begin{equation}\label{entropy-diss}
   \frac{d}{dt} \frac{\|f\|^2}{2} = \sprod{\sfL f}{f} \,,
\end{equation}
which is nonpositive as expected, but vanishes on the set of local equilibria, i.e.~it lacks the definiteness required for a Lyapunov
function.

In \cite{Schmeiser} a Lyapunov function, or modified entropy, $\sfH[f]$ has been proposed, which has the form
$$
  \sfH[f] := \frac{\|f\|^2}{2} + \delta \sprod{\sfA f}{f} \,,\qquad\mbox{with } \sfA := [1 + (\sfT\Pi)^\star (\sfT\Pi)]^{-1} (\sfT\Pi)^\star 
$$
and with a small parameter $\delta>0$ to be determined later. In \cite[Lemma 1]{Schmeiser} it has been shown that the operator 
norm of $\sfA$ is bounded by $\frac{1}{2}$, such that 
\begin{equation}\label{Hequiv}
  \frac{1 - \delta}{2} \|f\|^2 \le \sfH[f] \le \frac{1 + \delta}{2} \|f\|^2 \,,
\end{equation}
and $\sfH[f]$ is equivalent to $\|f\|^2$ for $\delta < 1$.

For solutions $f$ of \eqref{eq:abstract}, its time derivative is given by
\begin{equation}\label{dHdt}
\derivtintero \sfH[f] = \sprod{\sfL f}{f} - \delta \sprod{\sfA\sfT\Pi f}{f} - \delta\sprod{\sfA\sfT(1-\Pi)f}{f} + \delta \sprod{\sfT\sfA f}{f} 
  + \delta \sprod{\sfA\sfL f}{f} + \delta \sprod{\sfA f}{\sfL f} \,.
\end{equation}
From the definition of $\sfA$ it is clear that the property $\sfA = \Pi\sfA$ holds. On the other hand, the conservation of total mass
by the collision operator is equivalent to $\Pi\sfL = 0$, i.e.~$\Pi$ also projects to the nullspace of $\sfL^*$. As a consequence,
the last term above vanishes: 
$$
   \sprod{\sfA f}{\sfL f} = \sprod{\Pi\sfA f}{\sfL f} = \sprod{\sfA f}{\Pi\sfL f} = 0 \,.
$$ 
In \cite{Schmeiser} this property is the consequence of the assumption that $\sfL$ is symmetric, which does not hold here, 
as noted in the previous section. 

The first term on the right hand side of \eqref{dHdt} controls the microscopic part $(1-\Pi)f$ of the distribution function.
The second term is responsible for the macroscopic part:

\begin{lem}\label{lem:ATPi}
With the above notation, 
$$
   \sprod{\sfA\sfT\Pi f}{f} = \|\sfT\Pi g\|^2 + \|(\sfT\Pi)^* \sfT\Pi g\|^2 = \overline D \|\nabla_x \rho_g\|_2^2 + 
   {\overline D}^2\|\Delta_x \rho_g\|_2^2\,,
$$
where $\overline D = \sum_{i\le N_l} \eta_i\theta_i$, $\|\cdot\|_2$ = $\|\cdot\|_{L^2(\Omega)}$, and $g=\rho_g F$ is given by
$$
  g = (1+ (\sfT\Pi)^* \sfT\Pi)^{-1} \Pi f \,,\qquad\mbox{i.e.~}\rho_g\mbox{ solves}\quad \rho_g - \nc\overline D \Delta_x \rho_g = \rho_f\,.
$$ 
\end{lem}

\begin{proof}
The property $g = \Pi g$ is obvious from its definition.
We use the abbreviation $\mathcal{L} := (\sfT\Pi)^* \sfT\Pi$ and compute
\begin{eqnarray*}
   \sfA\sfT\Pi f = (1+\mathcal{L})^{-1} \mathcal{L} \Pi f = (1+\mathcal{L})^{-1} (1 + \mathcal{L} - 1)\Pi f = \Pi f - g = \mathcal{L}g \,.
\end{eqnarray*}
Therefore, using again the property $\sfA=\Pi\sfA$,
$$
  \sprod{\sfA\sfT\Pi f}{f} = \sprod{\sfA\sfT\Pi f}{\Pi f} = \sprod{\mathcal{L}g}{g + \mathcal{L}g} = \|\sfT\Pi g\|^2 + \|\mathcal{L}g\|^2 \,.
$$
A straightforward computation shows $\mathcal{L}g = -\nc\overline D \Delta_x \rho_g\,F$, completing the proof.
\end{proof}
As a consequence, the first two terms on the right hand side of \eqref{dHdt} provide the desired definiteness, since obviously
$\sprod{\sfA\sfT\Pi f}{f}\ge 0$ and $\sprod{\sfA\sfT\Pi f}{f} = 0\,\Rightarrow\, \rho_g=0\,\Rightarrow\,\Pi f = 0$. However, the remaining terms still need to be controlled.  

We start by using the diffusive macroscopic limit property \eqref{diff-macro}, implying
$$
   \sfT\sfA = -\sfT\Pi(1+\mathcal L)^{-1} \Pi\sfT = -(1-\Pi)\sfT\Pi(1+\mathcal L)^{-1} \Pi\sfT(1-\Pi) \,.
$$
In \cite[Lemma 1]{Schmeiser} it has been shown that the operator norm of $\sfT\sfA$ is bounded by 1 implying, together with 
the above,
\begin{equation}\label{bd-TA}
  |\sprod{\sfT\sfA f}{f}| \le \|(1-\Pi)f\|^2 \,.
\end{equation}

\begin{lem}\label{lem:AT(1-Pi)}
With the above notation,
$$
   |\sprod{\sfA\sfT(1-\Pi)f}{f}| \le C_1 \|(1-\Pi)f\| \sprod{\sfA\sfT\Pi f}{f}^{1/2} \qquad\mbox{with } 
   C_1 = \frac{1}{\overline D} \left( d(d+2)\sum_{i=1}^{N_l} \eta_i \theta_i^2\right)^{1/2}\,.
$$
\end{lem}
\begin{proof}
With $g$ as introduced in Lemma \ref{lem:ATPi} and with $\sfA^* = \sfT\Pi(1+\mathcal{L})^{-1} 
= \sfT(1+\mathcal{L})^{-1}\Pi$ we have
$$
   \sprod{\sfA\sfT(1-\Pi)f}{f} = - \sprod{(1-\Pi)f}{\sfT\sfA^*f} = -\sprod{(1-\Pi)f}{\sfT^2 g}
$$
and
$$
  (\sfT^2 g)_i = \sigma_i \,v\cdot\nabla_x (v\cdot\nabla_x \rho_g)F_i  \,.
$$
The estimate
$$
  \|\sfT^2 g\|^2 \le \sum_{i=1}^{N_l} \intv{|v|^4 F_i} \|\nabla_x^2 \rho_g\|_2^2 
  = d(d+2)\sum_{i=1}^{N_l} \eta_i \theta_i^2 \,\|\Delta_x \rho_g\|_2^2
$$
and an application of Lemma \ref{lem:ATPi} complete the proof.
\end{proof}

\begin{lem}\label{lem:AL}
With the above notation,
$$
   |\sprod{\sfA\sfL f}{f}| \le C_2 \|(1-\Pi)f\| \sprod{\sfA\sfT\Pi f}{f}^{1/2} \qquad\mbox{with } 
   C_2 = \left( 2N \max_{1\le j\le N} \sumi \frac{k_{ij}^2}{\eta_i} + 2 \max_{1\le i\le N} \left(\sumj k_{ji} \right)^2\right)^{1/2}\,.
$$
\end{lem}

\begin{proof}
We use (remembering $\sfL = \sfL(1-\Pi)$ and, from the preceding proof, $\sfA^* f = \sfT\Pi g$)
$$
  \sprod{\sfA\sfL f}{f} = \sprod{\sfL(1-\Pi)f}{\sfT\Pi g} \,,
$$
Lemma \ref{lem:ATPi}, and the boundedness of $\sfL$:
\begin{eqnarray*}
  \|\sfL f\|^2 &\le& 2 \sumi \intv{\int_\Omega \frac{M_i}{\eta_i} \left( \sumj k_{ij}\rho_j\right)^2 dx} 
  + 2 \sumi \intv{\int_\Omega \frac{f_i^2}{F_i} \left( \sumj k_{ji}\right)^2 dx} \\
  &\le& 2N \sumj \sumi \frac{k_{ij}^2}{\eta_i} \int_\Omega \rho_j^2 dx + 2 \max_{1\le i\le N} \left(\sumj k_{ji} \right)^2 \|f\|^2 \\
  &\le& \left( 2N \max_{1\le j\le N} \sumi \frac{k_{ij}^2}{\eta_i} + 2 \max_{1\le i\le N} \left(\sumj k_{ji} \right)^2\right) \|f\|^2 \,.
\end{eqnarray*}
\end{proof}

Collecting the results of Lemmas \ref{lem:micro}, \ref{lem:ATPi}, \ref{lem:AT(1-Pi)}, and \ref{lem:AL}, we obtain
$$
   \frac{d}{dt}\sfH[f] \le -(\lambda_m - \delta)\|(1-\Pi)f\|^2 - \delta \sprod{\sfA\sfT\Pi f}{f} 
   + \delta(C_1 + C_2)\|(1-\Pi)f\|\sprod{\sfA\sfT\Pi f}{f}^{1/2} \,.
$$
Thus, for 
$$
  \delta < \frac{4\lambda_m}{4+(C_1+C_2)^2} \,,
$$
we have 
\begin{equation}\label{dHdtle0}
   \frac{d}{dt}\sfH[f] \le - \lambda_\delta (\|(1-\Pi)f\|^2 + \sprod{\sfA\sfT\Pi f}{f}) 
\end{equation}
with
$$
   \lambda_\delta = \frac{1}{2} \left( \lambda_m - \sqrt{\lambda_m^2 - \delta( 4\lambda_m - 4\delta - \delta(C_1 + C_2)^2)}\right) >0 \,.
$$
This shows that $\sfH[f]$ is a Lyapunov function. It remains to obtain the decay rates.

\subsection*{Exponential decay on the torus (proof of Theorem \ref{thm:torus})}

For the case of the torus, i.e.~$\Omega = \mathbb{T}^d$, recalling the normalization to $\rho_\infty=0$ from the beginning of this section, with the notation of Lemma \ref{lem:ATPi} we have
$$
   \int_{\mathbb{T}^d} \rho_g \,dx = \int_{\mathbb{T}^d} \rho_f \,dx = 0 \,.
$$
The Poincar\'e inequality on the torus therefore provides \emph{macroscopic coercivity}, i.e.~there exists $\lambda_M>0$ such that
$$
  \|\sfT\Pi g\|^2 = \overline D \|\nabla_x \rho_g\|_2^2 \ge \lambda_M \|\rho_g\|_2^2 = \lambda_M \|g\|^2 \,.
$$
This is used in
\begin{eqnarray*}
  \|\Pi f\|^2 &=& \|g + \mathcal{L}g\|^2 = \|g\|^2 + 2\|\sfT\Pi g\|^2 + \|\mathcal{L}g\|^2 
    \le \left( \frac{1}{\lambda_M} + 2\right)\|\sfT\Pi g\|^2 + \|\mathcal{L}g\|^2 \\
    &\le& \left( \frac{1}{\lambda_M} + 2\right) \sprod{\sfA\sfT\Pi f}{f} \,.
\end{eqnarray*}
Combining this estimate with \eqref{Hequiv} and with \eqref{dHdtle0} gives
$$
  \frac{d}{dt} \sfH[f] \le - \frac{2\lambda_\delta \lambda_M}{(1+2\lambda_M)(1+\delta)} \sfH[f] \,,
$$
completing the proof of Theorem \ref{thm:torus} with 
$$
    \lambda = \frac{2\lambda_\delta \lambda_M}{(1+2\lambda_M)(1+\delta)} \,,\qquad C = \frac{1+\delta}{1-\delta} \,,
    \qquad \delta < \min\left\{ 1, \frac{4\lambda_m}{4+(C_1+C_2)^2}\right\} \,.
$$

\subsection*{Algebraic decay on the whole space (proof of Theorem \ref{thm:Rd})}

In the case $\Omega=\R^d$ macroscopic coercivity fails and is replaced by the Nash inequality \cite{Nash}
$$
   \|u\|_2^{2+4/d} \le \mathcal{C} \|\nabla u\|_2^2 \|u\|_1^{4/d} \qquad \forall\, u \in H^1(\R^d)\cap L^1(\R^d) \,.
$$
Noting that 
$$
    \int_{\R^d} \rho_g \,dx = \int_{\R^d} \rho_f \,dx = \textbf{M} \,,
$$
it implies
$$
  \|\sfT\Pi g\|^2 = \overline D \|\nabla_x \rho_g\|_2^2 \ge \kappa_M \|g\|^{2+4/d} \,,\qquad\mbox{with } 
  \kappa_M = \frac{\overline D}{\mathcal{C}\textbf{M}^{4/d}}  \,,
$$
and, thus,
\begin{eqnarray*}
  \|\Pi f\|^2 &\le& \kappa_M^{-\frac{d}{d+2}} \|\sfT\Pi g\|^{\frac{2d}{d+2}} + 2\|\sfT\Pi g\|^2 + \|\mathcal{L}g\|^2 \\
  &\le& \kappa_M^{-\frac{d}{d+2}}\sprod{\sfA\sfT\Pi f}{f}^{\frac{d}{d+2}} + 2 \sprod{\sfA\sfT\Pi f}{f} =: \Phi(\sprod{\sfA\sfT\Pi f}{f}) \,.
\end{eqnarray*}
Furthermore, by the properties of $\Phi$,
$$
   \|f\|^2 \le \Phi\left( \|(1-\Pi)f\|^2 + \sprod{\sfA\sfT\Pi f}{f} \right)
$$
Similarly to the case of the torus we now obtain from \eqref{dHdtle0} the differential inequality
$$
  \frac{d}{dt} \sfH[f] \le -\lambda_\delta \Phi^{-1}\left( \frac{2}{1+\delta} \sfH[f]\right) \,.
$$
The decay of $\sfH[f]$ can be estimated by the solution $z$ of the corresponding ODE problem
\[
  \frac{dz}{dt} =  - \lambda_\delta \Phi^{-1}\left( \frac{2z}{1+\delta} \right) \,,\qquad z(0) = \sfH[f_I] \,.
\]
By the properties of $\Phi$ it is obvious that $z(t)\to 0$ monotonically as $t\to\infty$, which implies that the same is true
for $\frac{dz}{dt}$. Therefore, there exists $t_0>0$ such that, in the rewritten ODE
\[
 \left(-\frac{2}{\lambda_\delta\kappa_M} \frac{dz}{dt}\right)^{\frac{d}{d+2}} -\frac{2}{\lambda_\delta} \frac{dz}{dt} 
 = \frac{2z}{1+\delta} \,,
\]
the second term is smaller than the first for $t\ge t_0$, implying the differential inequality
\[
   \frac{dz}{dt} \le -\kappa z^{1+2/d} \qquad\mbox{for } t\ge t_0\,,
\]
with an appropriately defined positive constant $\kappa$. Integration  gives
\[
 z(t) \le \left(\sfH[f_I]^{-2/d} + \frac{2\kappa}{d} t\right)^{-d/2} \,,
\]
completing the proof of Theorem \ref{thm:Rd}.

\section{The rigorous macroscopic limit (proof of Theorem \ref{thm:macro})}\label{sec:macro-limit}

The goal of this section is to justify the formal macroscopic limit $\eps\to 0$, carried out in Section \ref{sec:results} on the 
scaled equation \eqref{eq:eps}. The entropy dissipation relation \eqref{entropy-diss} now reads
\begin{equation}
\label{eq: test fepsilon}
\frac{\eps^2}{2}\derivtintero \|f_\eps\|^2  = \sprod{\sfL f_\eps}{f_\eps} \,.
\end{equation} 
Integration with respect to time and microscopic coercivity (Lemma \ref{lem:micro}) give
\begin{equation*}
\frac{\eps^2}{2}\|f_\eps(\cdot,\cdot,t)\|^2 + \lambda_m \int_0^t \|(1-\Pi)f_\eps(\cdot,\cdot,s)\|^2 ds \le \frac{\eps^2}{2} \|f_I\|^2 \,.
\end{equation*} 
From this relation we deduce that $f_\eps$ is bounded in $L^\infty(\R^+;\calH)$ and that $R_\eps = \frac{1}{\eps}(1-\Pi)f_\eps$ is
bounded in $L^2(\R^+;\calH)$, both uniformly as $\eps\to 0$. As a consequence of the former, $\rho_\eps = \intv{f_\eps}$ is
uniformly bounded in $L^\infty(\R^+;L^2(\R^d))$. Therefore there exist weak accumulation points $f_0$, $R_0$, and $\rho_0$ of, 
respectively, $f_\eps$, $R_\eps$, and $\rho_\eps$, satisfying $f_0(x,v,t) = \rho_0(x,t)F(v)$. These facts allow to pass to the limit in \eqref{kinetic/eps} and in the rescaled conservation law \eqref{mass-cons}, i.e.
$$
   \eps\partial_t f_\eps + \sfT f_\eps = \sfL R_\eps \,,\qquad 
   \partial_t \rho_\eps + \nabla_x\cdot \sum_{i=1}^{N_l} \intv{v R_{\eps,i}} = 0 \,,
$$
in the sense of distributions. The limiting equations are equivalent to the distributional formulation of the heat equation
\eqref{heat-eq} for $\rho_0$ with the initial condition $\rho_0(t=0) = \intv{f_I}$. The uniqueness of the solution of this problem
implies the weak convergence of $f_\eps$ to $\rho_0 F$. This completes the proof of Theorem \ref{thm:macro}.

\noindent {\bf Acknowledgments.}~~This work has been supported by the Austrian Science Fund, grants no. W1245 and F65. 
C.S. also acknowledges support by the Fondation Sciences
Math\'ematiques de Paris, and by Paris Science et Letttres. The authors are grateful for the hospitality at Sorbonne Universit\'e
and at Universit\'e Paris Dauphine.


\end{document}